\documentclass[a4paper,12pt]{amsart}
\usepackage{amsthm}
\usepackage{cases}
\usepackage{amsfonts}
\usepackage{amsmath}
\usepackage{amssymb}
\usepackage{mathtools}
\usepackage{amssymb,hyperref}

\DeclareMathOperator{\rank}{rank}

\DeclareMathOperator{\ord}{ord}

\DeclareMathOperator{\N}{\mathcal{N}}

\DeclareMathOperator{\vol}{vol}
\DeclareMathOperator{\ind}{ind}

\newcommand{\bbeta}{\boldsymbol{\beta}}
\newcommand{\balpha}{\boldsymbol{\alpha}}
\newcommand{\Q}{\mathbb{Q}}
\newcommand{\Qbar}{\overline{\Q}}
\newcommand{\C}{\mathbb{C}}
\newcommand{\R}{\mathbb{R}}
\newcommand{\Z}{\mathbb{Z}}
\newcommand{\F}{\mathbb{F}}
\newcommand{\A}{\mathcal{A}}
\newcommand{\B}{\mathcal{B}}
\newcommand{\D}{\mathcal{D}}
\newcommand{\E}{\mathcal{E}}
\newcommand{\T}{\mathcal{T}}
\newcommand{\OO}{\mathcal{O}}
\newcommand{\h}{\hat{h}}
\newcommand{\leg}[2]{\left(\frac{#1}{#2}\right)}

\usepackage{dsfont}
\newcommand{\One}[1]{{\mathds{1}}_{\{#1\}}}

\newtheorem{theorem}{Theorem}[section]
\newtheorem{lemma}[theorem]{Lemma}

\calclayout

\begin{document}

\title{Integral points on the congruent number curve}
\author{Stephanie Chan}
\email{ytchan@umich.edu}
\address{Department of Mathematics, University College London, Gower Street, London, WC1E~6BT, United Kingdom}

\begin{abstract}
We study integral points on the quadratic twists $\E_D:y^2=x^3-D^2x$ of the congruent number curve.
We give upper bounds on the number of integral points in each coset of $2\E_D(\Q)$ in $\E_D(\Q)$ and show that their total is $\ll (3.8)^{\rank \E_D(\Q)}$.
We further show that the average number of non-torsion integral points in this family is bounded above by $2$.
As an application we also deduce from our upper bounds that the system of simultaneous Pell equations
$aX^2-bY^2=d,\ bY^2-cZ^2=d$ for pairwise coprime positive integers $a,b,c,d$, has at most $\ll (3.6)^{\omega(abcd)}$ integer solutions.
\end{abstract}

\maketitle

\section{Introduction}
For any squarefree positive integer $D$, we consider the elliptic curve
\[\E_D:y^2=x^3-D^2x.\]
We are interested in the set of integral points on the curve, defined as 
\[\E_D(\Z)\coloneqq \left\{(x,y)\in\Z^2: y^2=x^3-D^2x\right\}.\]

Given an elliptic curve with Weierstrass equation $y^2=x^3+Ax+B$, Siegel~\cite{Siegelint} proved that there are only finitely many integral points, using techniques from the theory of Diophantine approximation.
Baker~\cite[p.45]{Baker} gave the first effective bound on the height of integral points: if an integral point $(x,y)$ exists, then 
\[|x|\leq \exp\big((10^6\max\{A,B\})^{10^6}\big).\]
Lang~\cite[p.140]{LangDioph} conjectured that the number of integral points on a (quasi)minimal Weierstrass equation of an elliptic curve should be bounded only in terms of its rank.
This was proven for elliptic curves with integral $j$-invariant~\cite[Theorem~A]{SilvermanSie} and for elliptic curves with bounded Szpiro ratio~\cite[Theorem~0.7]{HS}. The curves $\E_D$ satisfy both of these properties, and more specifically the theorems show that there exists some constant $C$ that is computable but not made explicit, such that
\begin{equation}\label{eq:SHS}
 \#\E_D(\Z)\ll C^{\rank \E_D(\Q)}.
\end{equation}
In \cite{GS}, Gross and Silverman presented a bound with explicit constants, which in our case gives $C$ of order $10^9$ in \eqref{eq:SHS}. The bound was improved in \cite{CLT}, which implies $C=24$ in \eqref{eq:SHS}. In \cite[Theorem 2]{CZ}, the bound in \cite[Theorem~0.7]{HS} was recovered by a different method and with some additional precision.
From a more general theorem by Helfgott and Venkatesh~\cite[Corollary 3.11]{HV}, we can deduce that 
\[\#\E_D(\Z)\ll C^{\omega(D)}(\log D)^2 (1.33)^{\rank \E_D(\Q)},\]
where $C$ is some absolute constant and $\omega(D)$ denotes the number of distinct prime factors of the integer $D$.
We obtain an upper bound in the form of \eqref{eq:SHS} with a smaller and explicit base, specifically for the curves $\E_D$.
\begin{theorem}\label{thm:posint}
We have
\[\#\E_D(\Z)\ll (3.8)^{\rank \E_D(\Q)}.\]
\end{theorem}
Therefore if we expect the rank to be uniformly bounded for all $\E_D(\Q)$ (as has been recently conjectured by various authors), then there would be a squarefree positive integer $D$ such that $\#\E_D(\Z)$ attains its maximum.

We proceed by partitioning $\E_D(\Z)$ into cosets of $2\E_D(\Q)$. 
For any $R\in \E_D(\Q)$, define 
\[
\mathcal{Z}_D(R)\coloneqq \E_D(\Z)\cap (R+2\E_D(\Q)).\]
We obtain an upper bound on the size of each $\mathcal{Z}_D(R)$ in terms of the rank of $\E_D(\Q)$:

\begin{theorem} \label{case:positive1} 
If $D$ is sufficiently large and $R\in \E_D(\Q)$ then 
 \[\#\mathcal{Z}_D(R)< 30+(1.89)^{r+19r^{1/3}},\]
where $r\coloneqq \rank \E_D(\Q)$.
\end{theorem}

Since $\#\E_D(\Q)/2\E_D(\Q)= 2^{2+\rank \E_D(\Q)}$,
Theorem~\ref{thm:posint} is immediate from Theorem~\ref{case:positive1}. 

Fix $\epsilon>0$. We partition $\mathcal{Z}_D(R)$ into the points with ``small'' $x$-coordinates,
\[
\mathcal{S}_{D}(R)\coloneqq \left\{P\in \mathcal{Z}_D(R):x(P)\leq D^{2(1+\epsilon)}\right\}
\]
and the points with ``large'' $x$-coordinates,
\[
\mathcal{L}_{D}(R)\coloneqq \left\{P\in \mathcal{Z}_D(R):x(P)> D^{2(1+\epsilon)}\right\},
\]
which we will bound by very different techniques.

\begin{theorem}[Points with large $x$-coordinates] \label{case:positive2} 
There exists some $\epsilon>0$ such that the following holds for any sufficiently large $D$ and $R\in \E_D(\Q)$. 
\begin{enumerate}
\item $\#\mathcal{L}_{D}(R)\leq 30$;
\item If the $abc$ conjecture holds, then $\mathcal{L}_D(R)=\varnothing$.
\end{enumerate}
\end{theorem}

We will complete the proof of Theorem~\ref{case:positive1} by showing that $\#\mathcal{S}_{D}(R)<(1.89)^{r+19r^{1/3}}$.

If $x(R)\leq D$ then we can improve the bound to $\#\mathcal{Z}_D(R)\leq 4$.
\begin{theorem}[Cosets with respect to points with very small $x$-coordinates]
\label{thm:elliptic}\leavevmode
\begin{enumerate}
\item \label{case:infty} 
$\mathcal{Z}_D(\mathcal O)=\varnothing$;
\item \label{case:ntors} 
$\mathcal{Z}_D((-D,0))=\{(-D,0)\}$ and $\mathcal{Z}_D((0,0))=\{(0,0)\}$;
\item \label{case:ptors} 
$\mathcal{Z}_D((D,0))$ contains $(D,0)$ and no more than one other pair $P,-P\in \E_D(\Z)$, given by $x(P)=(2v^2-1)D$, where $v+u\sqrt{D}$ is the fundamental solution of the equation $v^2-Du^2 = 1$;
\item \label{case:negative} 
If $R\in \E_D(\Q)$ and $-D<x(R)<0$, then $\mathcal{Z}_D(R)$ contains at most one pair $P,-P\in \E_D(\Z)$, except for the sets
\[\begin{split}
\{(-98, \pm 12376),\ (-1058, \pm 21896)\} &\text{ when } D=1254, \\
\text{and } \{(-5184,\pm 398664),\ (-7056,\pm 233772)\} &\text{ when } D=7585.
\end{split}\]
\end{enumerate}
\end{theorem}

The sets considered in Theorem~\ref{thm:elliptic}~\eqref{case:infty},~\eqref{case:ntors} contains no non-trivial integral points, and the upper bounds obtained in~\eqref{case:ptors},~\eqref{case:negative} are sharp. Indeed, on the curve $\E_6(\Q)$ of rank $1$, the distinct cosets $\mathcal{Z}_6(R)$ of integral points are $\{(-6,0)\}$, $\{(0,0)\}$, as well as
\[
 \{(-3,\pm 9)\},\ \{(-2,\pm 8)\},\ \{(6,0),\ (294,\pm 5040)\},\ \{(12,\pm 36)\},\ \{(18,\pm 72)\}.
\]

Ordering the curves $\E_D$ with increasing $D$, Heath-Brown~\cite[Theorem~1]{HBSelmer} showed that the moments of the $2$-Selmer of $\E_D$ are bounded. Together with~\eqref{eq:SHS}, this implies that the average size of $\E_D(\Z)$ is bounded. The boundedness of the average of $\#\E_D(\Z)$ was also proved by Alpoge~\cite{Alpoge}, but the upper bound was not explicitly evaluated.

Let $\D_N$ be the set of positive squarefree integers less than $N$. Define $\T_D$ to be the set of torsion points on $\E_D(\Q)$. It is standard that $\T_D=\{\OO,\ (0,0),\ (\pm D,0)\}$ (see for example \cite[Chapter~I, Proposition~17]{Koblitz}). 
Let $s_{2^{\infty}}(D)$ denote the $\Z_2$-corank of the $2$-power Selmer group of $\E_D(\Q)$, and $s_{2^k}(D)$ denote the $\F_2$-rank of the $2^k$-Selmer rank of $\E_D(\Q)$. Then $s_{2^{\infty}}(D)=\lim_{k\rightarrow\infty}s_{2^k}(D)$. Each $s_{2^k}(D)$ and hence also $s_{2^{\infty}}(D)$ provides an upper bound on the rank of $\E_D(\Q)$.
Heath-Brown~\cite{HBSelmer} notes that it can be derived from results of Cassels~\cite{Cassels4} and Birch and Stephens~\cite{BirchStephens}, that $s_2(D)$ is even for $D\equiv 1$, $2$ or $3\bmod 8$, and odd for $D\equiv 5$, $6$ or $7\bmod 8$. 
An elementary proof of this parity condition was given by Monsky~\cite[Appendix]{HBSelmer}. 
Furthermore, the $2^{k+1}$-Selmer group is computed from the kernel of the Cassels-Tate pairing on the $2^k$-Selmer group. Since the Cassels-Tate pairing is always skew-symmetric~\cite{Cassels3}, we have $s_{2^{k}}(D)\equiv s_{2^{k+1}}(D)\bmod 2$ for all $k$, so $s_{2^{\infty}}(D)$ and $s_2(D)$ are of the same parity. 
Smith~\cite[Corollary~1.2]{Smith2} recently claimed that
\[
\{D\in \D_N :s_{2^{\infty}}(D)\geq 2\}=o(N).
\]
It then follows that for $s\in\{0,1\}$, we have
\begin{equation}\label{eq:minconj}
\lim_{N\rightarrow\infty}\frac{1}{\#\D_N}\#\{D\in \D_N :s_{2^{\infty}}(D)=s\}=\frac{1}{2}.
\end{equation}
Since $\rank \E_D(\Q)\leq s_{2^{\infty}}(D)$, asymptotically at most half of the curves are of rank $1$, and density $0$ of curves are of rank $2$ or above.
This allows us to focus on curves with rank $0$ and $1$, hence we can find a better upper bound on the average.

\begin{theorem}\label{thm:avgint}
We have
\[\limsup_{N\rightarrow\infty}\frac{1}{\#\D_N}\sum_{D\in \D_N}\#(\E_D(\Z)\setminus\T_D)\leq 2.\]
If we further assume the $abc$ conjecture, 
the upper bound can be improved to $1$.
\end{theorem}
Note that non-torsion integral points come in pairs of $(x,\pm y)$.
The upper bound from Theorem~\ref{thm:avgint} comes from the possible existence of a pair of small points in the range $D^2/(\log D)^{12+\epsilon}<x<D^{2+\epsilon}$, and a pair of large points of size $x>\exp(\exp(\frac{23}{12}\sqrt{\log D}))$ left from an application of Roth's Theorem, which we are unable to eliminate on most curves of rank $1$.

We expect the order of $\sum_{D\in \D_N}\#(\E_D(\Z)\setminus\T_D)$ to be roughly $N^{1/2}$.
To obtain a lower bound, we attempt by counting a subset of integral points.
Suppose $u>v$ are squarefree positive coprime integers. 
Let $w$ be the squarefree part of $u^2-v^2$, so $u,v,w$ are pairwise coprime.
If $D=uvw$, then
$(u^2 w,u^2w^{3/2}\sqrt{u^2-v^2} )\in\E_D(\Z)$, since $w(u^2-v^2)$ is a square by the definition of $w$.
If $uv(u^2-v^2)<N$, then $D\in\D_N$, so counting the number of squarefree coprime positive integers $u,v$ in the range $v<u<N^{1/4}$, gives a lower bound of $\gg N^{1/2}$.

Now we give a heuristic on the maximum size of $\sum_{D\in \D_N}\#(\E_D(\Z)\setminus\T_D)$. The larger points $(x,y)\in\E_D(\Z)$ with $x>D^{2+\epsilon}$ can be removed by assuming the $abc$ conjecture as in Theorem~\ref{case:positive2}, so let's look at 
$D\in \D_N$ and $|x|<D^{2+\epsilon}$.
If $x=-j$, $j-D$, or $D+j$ for $1\leq j\leq D/2$ then $x^3-D^2x\approx jD^2$.
If $\frac{3}{2}D<x<N^3$ then $x^3-D^2x\approx x^3$.
Then we expect the number of pairs $(D,x)$ such that $x^3-D^2x$ is a square to be approximately
\[\sum_{\frac{1}{2}N\leq D<N} 
\left(\sum_{1\leq j\leq D/2} \left(\frac{1}{jD^2}\right)^{1/2} 
+ \sum_{\frac{3}{2}D<x<D^3} \left(\frac{1}{x^3}\right)^{1/2}\right)
\ll N^{1/2}.
\]

To prove Theorem~\ref{case:positive1}, we bound $\#\mathcal{S}_{D}(R)$ and $\#\mathcal{L}_D(R)$ separately.
We prove that $\#\mathcal{S}_{D}(R)$ is bounded above by
\[
\#\left\{P\in R+2\E_D(\Q):\h(P)\leq 2(1+\epsilon)\log D+O(1)\right\},\]
where $\h$ denotes the canonical height.
Then viewing $\E_D(\Q)$ as an $r$-dimensional Euclidean space, we apply sphere packing bounds to get an upper bound of $(1.89)^{r+19r^{1/3}}$ after fixing some appropriate $\epsilon$.

On the other hand, we show that $\#\mathcal{L}_D(R)$ is bounded by some constant depending only on $\epsilon$. 
Assume $x(R)>D$ and $R\notin\T_D+2\E_D(\Q)$, otherwise the result follows from Theorem~\ref{thm:elliptic}.
We first prove that points in $\mathcal{L}_{D}(R)$ obey a gap principle. Then, for points with larger heights in $\mathcal{L}_{D}(R)$, we apply Roth's theorem in a way that is similar to a classical argument of Siegel's Theorem, which also appeared in Alpoge's work~\cite{Alpoge}. Suppose $P=2Q+R\in \mathcal{Z}_D(R)$ and $x(P)$ is large, $P$ is close to the point at infinity. Let $K$ be the minimal number field containing the $x$-coordinates of all points in $\frac{1}{2}\E_D(\Q)$. If $4S=R$ and $2\tilde{Q}=Q$, where $\tilde{Q}, S\in \E_D(\Qbar)$, then $S$ and $\tilde{Q}$ are close together. Making this precise, we can show that $x(\tilde{Q})$ gives a $K$-approximation to $x(S)$ with exponent close to $8$. Roth's theorem show that there are finitely many such $\tilde{Q}$. In~\cite{Alpoge}, large integral points of the form $P=3Q+R$ were considered, where $Q,R$ are rational points on a general elliptic curve. The main difference of our approach is that we apply Roth's theorem over $K$ instead of $\Q$. Given a class in $\E_D(\Q)/n\E_D(\Q)$, the exponents of the $\Q$-approximations obtained from the argument in~\cite{Alpoge} are close to $\frac{1}{2}n^2$. If we had taken $n=2$, the exponent would be just under $2$ which would not be large enough to apply Roth's theorem. Applying the argument over $K$ instead gives a large enough exponent.

\section{Applications to other Diophantine equations}
Given positive integers $a$, $b$ and $c$, Bennett~\cite[Theorem~1.2]{Bcons} proved that there exists at most one set of three consecutive integers of the form $cZ^2$, $bY^2$, $aX^2$. In other words, the simultaneous equations
\[
aX^2-bY^2=1,\qquad bY^2-cZ^2=1, \qquad (X,Y,Z)\in\Z_{> 0}^3
\]
possess at most one solution. We can ask a more general question replacing the $1$ in the equations with an integer $d$.

\begin{theorem}\label{thm:simPell}
Let $a,b,c,d$ be pairwise coprime positive integers and set $D=abcd$. 
Then for any sufficiently large $D$, the system of equations \begin{equation}\label{eq:simPell}
aX^2-bY^2=d,\qquad bY^2-cZ^2=d
\end{equation}
has at most 
$15+(1.89)^{r+19r^{1/3}}\leq 
15+(3.58)^{\omega(D)+12\omega(D)^{1/3}}$ solutions $(X,Y,Z)\in\Z_{> 0}^3$, where $r\coloneqq \rank \E_D(\Q)$.
\end{theorem}

A different upper bound on the number of solutions to the system~\eqref{eq:simPell} depending only on $b$ and $d$, but involving implicit constants, can be deduced from \cite[p.41]{CZ}.  

We prove Theorem~\ref{thm:simPell} by relating the problem to our result in Theorem~\ref{case:positive1}. 
If we take $D=abcd$ and $x=ac(bY)^2$, then $x-D=ab(cZ)^2$ and $x+D=bc(aX)^2$.
Therefore $(ac(bY)^2,(abc)^2XYZ)\in \E_D(\Z)$.
The image of such a point under the injective homomorphism
\[\theta:\E_D(\Q)/2\E_D(\Q)\rightarrow \Q/(\Q^*)^2\times \Q/(\Q^*)^2\times \Q/(\Q^*)^2\]
given at non-torsion points by
\[(x,y)\mapsto 
(x-D,\ x,\ x+D),
\]
is $(ab,ac,bc)$.
If $P$ and $R$ are both integral points on $\E_D$ that correspond to solutions to~\eqref{eq:simPell}, then $P-R\in 2\E_D(\Q)$.
Moreover, $x(P)>0$ and $b^2\mid x(P)$.
Theorem~\ref{thm:simPell} is a corollary of Theorem~\ref{case:positive1} as $\pm P$ corresponds to the same solution for~\eqref{eq:simPell}.

More general forms of simultaneous Pell equations have been studied previously. For nonzero integers $a_1, a_2, b_1, b_2,u,v$, let $\N(a_1,a_2,b_1,b_2,u,v)$ denote the number of solutions to the system of equations
\[
a_1X^2-b_1Y^2=u,\qquad b_2Y^2-a_2Z^2=v\]
in positive integers $X,Y,Z$ such that $\gcd(X,Y,Z,u,v)=1$.
Theorem~\ref{thm:simPell} provides an upper bound to $\N(a,c,b,b,d,-d)$, where $a,b,c,d$ are pairwise coprime positive integers. 
Transforming the equations~\eqref{eq:simPell} by $X\mapsto aX$ and $Z\mapsto cZ$, we get $\N(a,c,b,b,d,-d)=\N(1,1,ab,bc,ad,-cd)$.
Assuming $a,b$ are distinct positive integers and $-av\neq bu$,
Bennett~\cite[Theorem~2.1]{BPell} showed that 
\[\N(1,1,a,b,u,v)\ll 2^{\min\{\omega(u),\omega(v)\}}\log(|u|+|v|).\] This implies $\N(a,c,b,b,d,-d)\ll 2^{\omega(d)\min\{\omega(a),\omega(c)\}}\log((a+c)d)$. 
Bugeaud, Levesque and Waldschmidt~\cite[Th\'{e}or\`{e}me~2.2]{BLW} gave the bound
\[\N(a_1,a_2,b_1,b_2,u,v)\leq 2+2^{3996(\omega(a_1a_2uv)+1)}.\] Translating to our case, this gives an upper bound
$\N(a,c,b,b,d,-d)\leq 2+2^{3996(\omega(acd^2)+1)}$.

Theorem~\ref{thm:simPell} can also provide an upper bound to a different Diophantine equation.
In 1942, Ljunggren~\cite{Ljunggren} showed that for a fixed integer $d$, the equation
\[X^4-dY^2=1,\qquad (X,Y)\in\Z_{>0}^2,\] 
has at most two solutions, through a study of units in certain quadratic and biquadratic fields.
More recently, Bennett and Walsh~\cite{BW} used the
theory of linear forms in logarithms of algebraic numbers, to show that for squarefree positive integers $b,d\geq 2$, the equation 
\[b^2X^4-dY^2=1,\qquad (X,Y)\in\Z_{>0}^2,\]
has at most one solution.
We prove the following as a corollary to Theorem~\ref{thm:simPell}.
\begin{theorem}\label{thm:sim4}
Let $A,B,C$ be pairwise coprime positive squarefree integers.
Then there are $\ll 2^{\omega(AB^2C^2)}$
integral solutions $(X,Y)$ to 
\[A^2X^4-BY^2=C^2.\]
\end{theorem}
\begin{proof}
Let $g\coloneqq \gcd(X,C)$.
Observe that
\[\left(Ag\left(\frac{X}{g}\right)^2-\frac{C}{g}\right)\left(Ag\left(\frac{X}{g}\right)^2+\frac{C}{g}\right)=B\left(\frac{Y}{g}\right)^2.\]
The factors on the left hand side have common factor $1$ or $2$. Therefore we can write
\begin{equation}\label{eq:sim4}
Ag\left(\frac{X}{g}\right)^2-\frac{C}{g}
=B_1Y_1^2\text{ and }
Ag\left(\frac{X}{g}\right)^2+\frac{C}{g}=B_2Y_2^2,
\end{equation}
where $B_1$ and $B_2$ are positive integers such that $B_1B_2=B$ or $4B$, and $Y_1Y_2g=Y$.
Now applying Theorem~\ref{thm:simPell}, the system of equations~\eqref{eq:sim4} has $\ll 2^{\omega(ABC)}$ solutions.
There are $2^{\omega(C)}$ choices of $g\mid C$ and $\ll 2^{\omega(B)}$ choices of pairs $(B_1,B_2)$. This proves Theorem~\ref{thm:sim4}.
\end{proof}

\subsection{The $abc$ conjecture}
In Theorem~\ref{case:positive2}, if we are allowed to assume the $abc$ conjecture, we can show that there exists some $\epsilon$ (determined by the conjecture) such that the set $\mathcal{L}_D(R)$ is empty when $D$ is sufficiently large. 
The $abc$ conjecture states that for every $\epsilon>0$, for any pairwise coprime positive integers $a$, $b$, $c$, with $a+b=c$, we have 
\[c\ll_{\epsilon}\prod_{p\mid abc}p^{1+\epsilon}.\]
Suppose that $(x,y)\in \E_D(\Z)$ and $x>0$.
Let $g=\gcd(x,D)$. Dividing $y^2=x^3-D^2x$ by $xg^2$ and rearranging, we have
\[\left(\frac{D}{g}\right)^2+\frac{y^2}{xg^2}=\left(\frac{x}{g}\right)^2.\]
Assuming the $abc$ conjecture,
\begin{equation}\label{eq:abcE}
\left(\frac{x}{g}\right)^2\ll_{\epsilon}
\prod_{p\mid \left(\frac{D}{g}\right)^2\left(\frac{x}{g}\right)^2\frac{y^2}{xg^2}}p^{1+\epsilon}.
\end{equation}
If
$p\mid(\frac{D}{g})^2(\frac{x}{g})^2\frac{y^2}{xg^2}$,
then
$p\mid(\frac{D}{g})(\frac{x}{g})\frac{y^2}{xg^2}=\frac{Dy^2}{g^4}$.
By construction $g\mid D$ and $g^3\mid y^2$. Since $g$ is squarefree, so $g^2\mid y$.
Therefore
$p\mid \frac{Dy}{g^2}$.
Putting this back to~\eqref{eq:abcE},
\[\left(\frac{x}{g}\right)^2\ll_{\epsilon}\left(\frac{Dy}{g^2}\right)^{1+\epsilon}
<\left(\frac{Dx^{3/2}}{g^2}\right)^{1+\epsilon}
.\]
Then for $\epsilon<\frac{1}{15}$, since $g\leq D$,
\[x\ll_{\epsilon}
\left(\frac{D^{2(1+\epsilon)}}{g^{4\epsilon}}\right)^{1-3\epsilon}
\leq 
D^{2\left(\frac{1+\epsilon}{1-3\epsilon}\right)}
<D^{2(1+5\epsilon)}
.\]
This proves the last assertion in Theorem~\ref{case:positive2}.

\section{Height estimates}
Notice that if $(x,y)\in \E_D(\Q)$, then either $x\geq D$ or $-D\leq x\leq 0$.
For $\alpha\in\Qbar$, define height functions
$H(\alpha)\coloneqq \prod_{v}\max\{1,|\alpha|_v\}$ and $h(\alpha)=\log H(\alpha)=\sum_{v}\log^+|\alpha|_v$, where $v$ is taken over the set of places of $\Q(\alpha)$ and $\log^+$ is a function on the positive real numbers, defined as
$\log^+ t=\max\{0,\log t\}$.
For any point $P\in \E_D(\Qbar)$, define $H(P)=H(x(P))$, denote the (Weil) height by $h(P)\coloneqq h(x(P))$ and the canonical height by 
\[\h(P)\coloneqq \lim_{n\rightarrow \infty}\frac{h(nP)}{n^2}.\footnote{This is sometimes defined with an extra factor of $\frac{1}{2}$ in literature.}\]

\begin{lemma}\label{lemma:doublegcd}
Let $P\in \E_D(\Q)$ be a non-torsion point. Write $x(2P)=\frac{r}{s}$, where $r$ and $s$ are coprime integers and $s>0$. Then 
$\gcd(r,D)=1$.
\end{lemma}
\begin{proof}
Suppose $P\in \E_D(\Q)$, then $\theta(2P)=(1,1,1)$, so write
\[x(2P)=\frac{r^2}{s^2},\qquad x(2P)-D=\frac{u^2}{v^2},\]
where $r,s,u,v\in\Z$ and $\gcd(r,s)=\gcd(u,v)=1$.
Combining gives
\[r^2v^2-u^2s^2=Dv^2s^2.\]
We see that $v=s$, since $\gcd(r,s)=\gcd(u,v)=1$. Rewriting 
\[r^2-u^2=Ds^2.\]
Since $\gcd(r,s)=\gcd(u,s)=1$ and $D$ is squarefree, $\gcd(r,D)=1$.
\end{proof}

We prove for points on $\E_D(\Q)$ that the Weil height and the canonical height are close together.
\begin{lemma}
Let $P=(x,y)\in \E_D(\Q)\setminus \{\OO,(0,0)\}$. Write $x=\frac{r}{s}$, where $r$ and $s$ are coprime integers and $s>0$.
If $x\geq D$, then
\begin{equation}\label{eq:height}
-\log |\gcd(r,D)|-2\log 2
\leq \h(P)-h(P)
\leq -\log |\gcd(r,D)|+\frac{2}{3}\log 2.
\end{equation}
If $-D\leq x<0$, then
\begin{equation}\label{eq:heightn}
\log \left|\frac{D}{\gcd(r,D)}\right|-\log^+|x|-2\log 2
\leq \h(P)-h(P)
\leq\log \left|\frac{D}{\gcd(r,D)}\right| -\log^+|x|+\frac{2}{3}\log 2.
\end{equation}
In particular, 
\begin{equation}\label{eq:height_sq}
-2\log 2
\leq 4\h(P)-h(2P)=\h(2P)-h(2P)
\leq \frac{2}{3}\log 2.
\end{equation}
\end{lemma}
\begin{proof}
Focusing on the $h(2^nP)$ terms in the limit defining $\h(P)$, we can express the canonical height as a telescoping series
\begin{equation}\label{eq:tele}
\h(P)=h(P)-\sum_{n=0}^{\infty}\frac{1}{4^n}\left(h(2^nP)-\frac{1}{4}h(2^{n+1}P)\right).\end{equation}

Consider a point $P\in \E_D(\Q)\setminus \{\OO,(0,0)\}$. 
Write $x(P)=\frac{r}{s}$, where $r,s$ are coprime integers and $s>0$.
Then
\[x(2P)=\frac{(r^2+D^2s^2)^2}{4rs(r-Ds)(r+Ds)}.\]
If an odd prime $p$ divides both $(r^2+D^2s^2)^2$ and $4rs(r-Ds)(r+Ds)$, then since $\gcd(r,s)=1$, $p$ divides both $r$ and $D$.
If $r^2+D^2s^2$ is even,
either $r,D,s$ are all odd, or $r,D$ are even and $s$ is odd.
The first case implies that $r^2+D^2s^2\equiv 2\bmod 8$, so $4\parallel (r^2+D^2s^2)^2$.
The second case note that $D$ is squarefree so $2\parallel D$. If $2\parallel r$ we have $4\cdot 2^{4}\parallel (r^2+D^2s^2)^2$, otherwise the $2^{4}\parallel (r^2+D^2s^2)^2$.

Therefore
\[\gcd\left((r^2+D^2s^2)^2,4rs(r-Ds)(r+Ds)\right)=(\gcd(r,D))^4\text{ or }4(\gcd(r,D))^4.\]
Since $x(2P)>D$, we have
\begin{equation}\label{eq:doubleh}
\begin{split}
h(2P)
&=\log(r^2+D^2s^2)^2-\log\gcd\left((r^2+D^2s^2)^2,4rs(r-Ds)(r+Ds)\right)\\
&=2\log(r^2+D^2s^2)-4\log\gcd(r,D)-\One{s\text{ odd}}\One{\ord_2{r}=\ord_2{D}} 2\log 2.
\end{split}
\end{equation}

We first prove~\eqref{eq:height}. Suppose $x(P)\geq D$.
Then
\begin{equation}\label{eq:diff}
h(P)-\frac{1}{4}h(2P)
 =-\frac{1}{2}\log\left(1+\frac{D^2s^2}{r^2}\right)
+\log\gcd(r,D)+\One{s\text{ odd}}\One{\ord_2{r}=\ord_2{D}} 2\log 2.
\end{equation}
Apply~\eqref{eq:diff} to~\eqref{eq:tele}.
 Then 
\[0<\log\left(1+\frac{D^2s^2}{r^2}\right)<\log 2.\]
Writing $x(2^nP)=\frac{r(2^nP)}{s(2^nP)}$ in lowest term, we know from Lemma~\ref{lemma:doublegcd} that $\gcd(r(2^nP),D)=1$ for any $n\geq 1$. 
The conditions $2\nmid s$ and $\ord_2{r}=\ord_2{D}$ can only hold simultaneously at most once in the sequence $2^nP$. For if $s,r,D$ are all odd, then subsequent terms would have even $s$. On the other hand, since the $x$-coordinates of double points must be squares, $2\parallel r$ can only happen in the first term. 
Noting that $\sum_{n=0}^{\infty}\frac{1}{4^n}=\frac{4}{3}$, we get~\eqref{eq:height}.

For~\eqref{eq:heightn}, suppose instead $-D<x(P)<0$. Then from~\eqref{eq:doubleh}, we have
\begin{multline*}
h(P)-\frac{1}{4}h(2P)\\
=\One{r>s}\log |x|-\frac{1}{2}\log\left(1+\frac{r^2}{D^2s^2}\right)
-\log \left|\frac{D}{\gcd(r,D)}\right|
+\One{s\text{ odd}}\One{\ord_2{r}=\ord_2{D}} 2\log 2.
\end{multline*}
Apply this to~\eqref{eq:tele}.
Similar to the argument for~\eqref{eq:height}, but here instead 
\[0<\log\left(1+\frac{r^2}{D^2s^2}\right)<\log 2,\]
we get~\eqref{eq:heightn}.

Finally~\eqref{eq:height_sq} follows from~\eqref{eq:height} and Lemma~\ref{lemma:doublegcd}. 
\end{proof}

Estimates equivalent to~\eqref{eq:height_sq} were obtained in~\cite[Section~2]{BST} by analysing the local height functions specifically for $\E_D$. The inequalities~\eqref{eq:height},\eqref{eq:heightn} with larger constant terms can be obtained via a study of local heights by applying theorems for general elliptic curves~\cite[Theorem~4.1, Theorem~5.4]{SilvermanHeight},
and~\cite[Theorem~5.2]{SilvermanAdd}. 

For general algebraic points on $\E_D$, we obtain the following estimate by applying~\cite[Equation(3)]{SilvermanHeight}, noting that the discriminant of $\E_D$ is $\Delta_D=(2D)^6$ and $j$-invariant is $1728$.
\begin{lemma}
Any $P\in \E_D(\Qbar)$ satisfies
\begin{equation}\label{eq:heightnf}
|\h(P)-h(P)|<\log D+4.6.\end{equation}
\end{lemma}

Since $\h(2P)=h(2P)+O(1)$ by~\eqref{eq:height_sq},
we have
\begin{equation}\label{eq:doublelb}
\h(2P)=h(2P)-2\log 2\geq \log D-2\log 2.
\end{equation}
Therefore for any $P\in \E_D(\Q)\setminus\T_D$,
\begin{equation}\label{eq:Lang}
\h(P)\geq \frac{1}{4}\log D-\frac{1}{2}\log 2.
\end{equation}
The equation~\eqref{eq:Lang} is a version of Lang's conjecture, which says that the canonical height of a non-torsion point on
an elliptic curve should satisfy
\[\h(P)\gg\log |\Delta|,\]
where $\Delta$ is the discriminant of the elliptic curve. This conjecture was proven for elliptic curves with integral $j$-invariant~\cite{Silvermanheightlb}, for elliptic curves which are twists~\cite{Silvermanfnlb}, and for elliptic curves with bounded Szpiro ratio~\cite{HS}. The curves $\E_D$ are in all three of these categories, as remarked in~\cite{BST}.
The bound~\eqref{eq:Lang} for curves $\E_D$ with the explicit constant factor $\frac{1}{4}$ was first given in~\cite[(11)]{BST}.

\section{Bounding small points via spherical codes}\label{sec:sphere}
In this section we prove the following lemma, which gives the upper bound of $\#\mathcal{S}_{D}(R)$ for Theorem~\ref{case:positive1}.
\begin{lemma}\label{lemma:smallptbd}
Suppose $R\in \E_D(\Q)$ with $x(R)>D$. Let $\epsilon<\frac{1}{650}$. Then for any sufficiently large $D$ we have
\[\#\{P\in R+2\E_D(\Q): \h(P)\leq 2(1+\epsilon)\log D\}<(1.89)^{r+19r^{1/3}}.\]
\end{lemma}

We first show how Lemma~\ref{lemma:smallptbd} implies the required upper bound of $\#\mathcal{S}_{D}(R)$. This follows from the claim that 
\[\mathcal{S}_{D}(R)\subseteq \{P\in R+2\E_D(\Q): \h(P)\leq 2(1+\epsilon+o(1))\log D\} .\]
Suppose $P\in \mathcal{S}_{D}(R)$, we need to show that $h(P)\leq 2(1+\epsilon)\log D$ implies $\h(P)\leq 2(1+\epsilon)\log D+O(1)$. 
If $x(P)\geq D$, from \eqref{eq:height} we have  $\h(P)\leq h(P)+O(1)\leq 2(1+\epsilon)\log D+O(1)$.
If $-D\leq x(P)\leq 0$,  from \eqref{eq:heightn} we have $\h(P)\leq h(P)+\log D+O(1)\leq  2\log D 
+O(1)$.

We now turn to the proof of Lemma~\ref{lemma:smallptbd}. We know that the the canonical height of the difference between any two distinct points in $R+2\E_D(\Q)$ is at least $\log D+O(1)$ from equation~\eqref{eq:doublelb}.
Viewing $R+2\E_D(\Q)$ as a Euclidean space $\R^r$ of dimension $r$, we can bound the number of points by the maximum number of spheres of radius $\frac{1}{2}\sqrt{\log D}+O(1)$ with centres lying inside a sphere $S_R^{r-1}$ of radius $R=\sqrt{2(1+\epsilon)\log D}$.
This is similar to the covering argument used in \cite{BZ} to obtain bounds on the number of rational points with bounded height on elliptic curves with full rational $2$-torsion.

A \emph{spherical code} in dimension $r$ with minimum angle $\theta$ is a set of points on the unit sphere $S_1^{r-1}$ in $\mathbb{R}^r$ with the property 
that no two points subtend an angle less than $\theta$ at the origin.
Let $A(r,\theta)$ denote the greatest size of such a spherical code.

We can obtain an upper bound in terms of the function $A$ via a classical argument (see for example the proof of (2.1) in~\cite{CohnZhao}).
Project the sphere centres in $S_R^{r-1}$ onto the upper hemisphere of $S_R^r$ orthogonally to the hyperplane. The projections of the sphere centres are still at least distance $\sqrt{\log D}+O(1)$ apart, and thus separated by angles of at least $\theta$, where $\sin\frac{\theta}{2}=\frac{1}{2\sqrt{2(1+\epsilon)}}-o(1)$.
Therefore the number of small points is bounded above by $\leq A(r+1,\theta)$.

\subsection{For large dimensions}

Kabatiansky and Levenshtein proved the following upper bound on $A(r,\theta)$.
\begin{theorem}[{\cite[(52)]{KabLev}}]
Let $r\geq 3$ and $\alpha=\frac{r-3}{2}$,
and let $t_{k}^{\alpha}$ be the largest root of 
\[P_k^{\alpha}(t)=\frac{1}{2^k}\sum_{i=0}^k\binom{k+\alpha}{i}\binom{k+\alpha}{k-i}(t+1)^i(t-1)^{k-i}.\]
Take any $k$ such that
$\cos\theta\leq t_{k}^{\alpha}$.
Then
\[A(r,\theta)\leq \frac{4}{1-t_{k+1}^{\alpha}}\binom{k+r-2}{r}.\]
\end{theorem}

From the proof of~\cite[Lemma~4]{KabLev}, we know that
\[\tau_k-\frac{2\pi^{2/3}}{((k+\alpha)(k+\alpha+1)\tau_k)^{1/3}}\leq t_k^{\alpha}\leq \tau_k,\]
where
\[\tau_k=\sqrt{1-\frac{\alpha^2-1}{(k+\alpha)(k+\alpha+1)}}.\]
Therefore if we take $k$ such that
\begin{equation}\label{eq:tkrange}
\cos\theta\leq \tau_k-\frac{2\pi^{2/3}}{((k+\alpha)(k+\alpha+1)\tau_k)^{1/3}},\end{equation}
and since
$t_{k+1}^{\alpha}\leq \tau_{k+1}$,
we have
\[A(r,\theta)\leq \frac{4}{1-\tau_{k+1}}\binom{k+r-2}{r}.\]
Take $\theta$ such that $\sin\frac{\theta}{2}=\frac{1}{2\sqrt{2(1+\epsilon)}}-o(1)$.
Let 
\[N=\frac{1-\sin\theta}{2\sin\theta}
=\frac{2(1+\epsilon)}{\sqrt{7+8\epsilon}}-\frac{1}{2}+o(1),\]
so that
\[\tau_k\rightarrow \sqrt{1-\frac{1}{(2N+1)^2}}\]
as $k\rightarrow\infty$ and $\frac{k}{r}\rightarrow N$.

Fix some
\begin{equation}\label{eq:crange}
C^3> 16\pi^2 N(N+1)(2N+1)^5
\end{equation}
Take
$k-2=\left\lfloor rN+Cr^{1/3}\right\rfloor$,
so~\eqref{eq:tkrange} is satisfied for large enough $r$.
 By Stirling's formula, we have
\[\binom{k+r-2}{r}\leq 
\frac{e}{2\pi}\frac{(k+r-2)^{k+r-\frac{3}{2}}}{r^{r+\frac{1}{2}}(k-2)^{k-2+\frac{1}{2}}}
\leq 
\frac{e}{2\pi\sqrt{r}}
\left(1+\frac{1}{N}\right)^{Cr^{1/3}+\frac{1}{2}}
\left(\frac{(1+N)^{1+N}}{N^N}\right)^r
\]
for large enough $r$.
Therefore for large enough $r$, we have the upper bound
\begin{equation}\label{eq:Abd}
A(r,\theta)<\left(\frac{(1+N)^{1+N}}{N^N}\right)^{r+\frac{\log (1+N)-\log N}{(1+N)\log (1+N)-N\log N}Cr^{1/3}},
\end{equation}
taking some small $C$ in the range~\eqref{eq:crange}.

We can now prove Lemma~\ref{lemma:smallptbd} for $r\geq 2000$.
Take $\epsilon<\frac{1}{650}$ and $C=\frac{189}{25}$, so~\eqref{eq:crange} is satisfied. 
Then we can rewrite the bound in~\eqref{eq:Abd} as
$A(r,\theta)<(1.89)^{r+19r^{1/3}}$.

\subsection{For small dimensions}
To prove Lemma~\ref{lemma:smallptbd}, it remains to check the same bound holds for $r<2000$.
The two following bounds, obtained respectively by Rankin and Shannon, are weaker asymptotically when $r\rightarrow\infty$ but are better bounds for small $r$.

\begin{theorem}[{\cite[Theorem~2]{Rankincaps}}]\label{theorem:Rankin}
If $0<\theta<\frac{\pi}{4}$ and $\sin\beta=\sqrt{2}\sin\theta$, then
\[
\begin{split}
A(r,\theta)
&\leq 
\frac{\sqrt{\pi}\Gamma(\frac{r-1}{2})\sin\beta\tan\beta}
{2\Gamma(\frac{r}{2})\int_0^{\beta}\sin^{r-2} x(\cos x -\cos\beta)dx}\\
&\leq 
\frac{2\sqrt{\pi}\Gamma(\frac{r+3}{2})\cos\beta}
{\Gamma(\frac{r}{2})\sin^{r-1}\beta(1-\frac{3}{r+3}\tan^2 \beta)}
\sim
\frac{\sqrt{\frac{1}{2}\pi r^3\cos 2\theta}}
{(\sqrt{2}\sin\theta)^{r-1}}.
\end{split}\]
\end{theorem}

\begin{theorem}[{\cite[(21),(27)]{Shannon}}]\label{theorem:Shannon}
Suppose $0<\theta<\frac{\pi}{2}$. Then
\[A(r,\theta)\leq 
\frac{\sqrt{\pi}\Gamma(\frac{r-1}{2})}{
\Gamma(\frac{r}{2})\int_{0}^{\theta}\sin^{r-2}x dx}
\leq
\frac{2\sqrt{\pi}\Gamma(\frac{r+1}{2})\cos \theta}{\Gamma(\frac{r}{2})\sin^{r-1}\theta(1-\frac{1}{r}\tan^2\theta)}
\sim \frac{\sqrt{2\pi r}\cos \theta}{\sin^{r-1}\theta}.\]
\end{theorem}

Evaluating the bounds in Theorems~\ref{theorem:Rankin} and~\ref{theorem:Shannon} for $r< 2000$ proves Lemma~\ref{lemma:smallptbd} in those cases.

\section{Repulsion between medium points}
Suppose $P=(X,Y)$ and $R=(x,y)$ are integral points in the same coset of $2\E_D(\Q)$.
Assume $D^{2(1+\epsilon)}<x<X$ and $Yy>0$. 
Suppose $P=2Q+R$ for some $Q\in \E_D(\Q)$.
Replacing $Q$ with one of $Q+(0,0)$, $Q+(D,0)$, $Q+(-D,0)$ if necessary, we can assume $x(Q)>(1+\sqrt{2})D$. 
\begin{lemma}\label{lemma:sumlb}
Suppose $P_1=(x_1,y_1),\ P_2=(x_2,y_2)\in \E_D(\Q)$ and $D<x_1\leq x_2$.
Let $B=\frac{x_2}{x_1}$ and $\mu=\frac{D^2}{x_1^2}$.
Then 
\begin{equation}\label{eq:sumlb}
\left(\frac{B+\mu}{\sqrt{B^2-\mu}+\sqrt{B(1-\mu)}}\right)^2 x_1
\leq
x(P_1+ P_2)
\leq
\left(\frac{\sqrt{B^2-\mu}+\sqrt{B(1-\mu)}}{B -1}\right)^2x_1
.
\end{equation}
Moreover \begin{equation}\label{eq:sumlbsim}
\frac{1}{4} x(P_1)
\leq
x(P_1+ P_2)
\leq
\left(\frac{2B}{B -1}\right)^2
 x(P_1)
.
\end{equation}
\end{lemma}
\begin{proof}
Assuming $y_1y_2>0$,
\[
x(P_1\pm P_2)=
\left(\frac{\sqrt{B^2-\mu}\mp\sqrt{B(1-\mu)}}{B -1}\right)^2x_1
=
\left(\frac{B+\mu}{\sqrt{B^2-\mu}\pm\sqrt{B(1-\mu)}}\right)^2x_1.\]
This shows that $x(P_1+P_2)$ have to attain the lower bound in \eqref{eq:sumlb} when $y_1y_2>0$, and attain the upper bound in \eqref{eq:sumlb} when $y_1y_2<0$.

Noting that $B\geq 1$ and $\mu<1$ by assumption, we deduce \eqref{eq:sumlbsim} from \eqref{eq:sumlb}.
\end{proof}

We now show that $x(Q+R)$ is properly bounded away from $D$.
 If $(1+\sqrt{2})D<x(Q)<4(1+\sqrt{2})D$, since we assumed $x(R)>D^{2(1+\epsilon)}$, by the lower bound in~\eqref{eq:sumlb} we have
$x(Q+R)>\frac{3}{4}(1+\sqrt{2})D$ for large enough $D$.
If $x(Q)\geq 4(1+\sqrt{2})D$, then $x(Q+R)\geq (1+\sqrt{2})D$ by the lower bound in~\eqref{eq:sumlbsim}.

\begin{lemma}\label{lemma:dblbd}
Suppose $Q\in \E_D(\Qbar)\setminus\{\OO,(0,0),(\pm D,0)\}$ and $x(Q)\in \R$.
Then
\[ x(2Q)\geq \frac{1}{4} |x(Q)|.\]
Moreover, if in addition $x(Q)\geq\frac{1}{\delta} D$ for some $\delta>1$, then
\[\frac{1}{4}x(Q)\leq x(2Q)\ll_{\delta} x(Q).\]
\end{lemma}
\begin{proof}
If $x(Q)> D$, this follows immediately from the formula
\[
x(2Q)
=\frac{\left(1+(\frac{D}{x(Q)})^2\right)^2}{4\left(1-(\frac{D}{x(Q)})^2\right)}x(Q)
.\]
If $-D<x(Q)<0$, then it is straightforward to check that $x(2Q)>D$, so  $x(2Q)>|x(Q)|$.
\end{proof}

Trivially 
$h(Q)\geq \log x(Q)$ and $h(Q+R)\geq \log x(Q+R)$. 
By Lemma~\ref{lemma:dblbd} and the lower bounds on $x(Q)$ and $x(Q+R)$, we have $x(Q)\gg x(2Q)=x(P-R)$ and $x(Q+R)\gg x(2Q+2R)=x(P+R)$. Also
$
x(P\pm R)
\gg
x$ by Lemma~\ref{lemma:sumlb}.
Putting together we have
$
h(Q),h(Q+R)\geq \log x+O(1)$.
Now apply~\eqref{eq:height} to $P$ and $P-R$, then to $Q$ and $Q+R$, it follows that
\[\log X+\log x+O(1)\geq \h(P)+\h(R)=2\h(Q)+2\h(Q+R)\geq 4\log x-4\log D+O(1).\]
Rearranging gives
\begin{equation}\label{eq:repulsion}\log X\geq 3\log x-4\log D+O(1).
\end{equation}

\section{Large integral points giving Diophantine approximations}\label{sec:largept}
Fix an embedding $\iota:\Qbar\hookrightarrow\C$ and write $|\cdot|$ for the corresponding absolute value.
In this section we will prove the following lemma.
\begin{lemma}\label{lemma:Roth}
Suppose $P\in \E_D(\Z)$ such that $P=4\tilde{Q}+R$, for some $\tilde{Q}\in \frac{1}{2}\E_D(\Q)$ and $R\in \E_D(\Q)\setminus 2\E_D(\Q)$.
Assume that $x(R)>D$ and $h(P)>\max\{\frac{1}{\lambda}h(R),\frac{1}{\delta}\log D\}$, where $0<\delta\leq \lambda<\frac{1}{1000}$.
Take $S\in \{\tilde{S}\in \E_D(\Qbar):4\tilde{S}=R\}$ such that $|x(\tilde{Q})-x(S)|$ is minimum.
Then
\[
\frac{\log|x(\tilde{Q})-x(S)|}{h(\tilde{Q})}\leq -8\cdot \frac{1
-63\lambda-418\delta}{(1+\sqrt{\lambda})^2+16\delta}+o(1)
.
\]
\end{lemma}

Suppose $P,\tilde{Q},R$ satisfies the assumptions of Lemma~\ref{lemma:Roth}.
Since $2\tilde{Q}\in \E_D(\Q)$, we have $x(\tilde{Q})\in \R$. Then 
\begin{equation}\label{eq:tQsize}
x(\tilde{Q})\ll x(4\tilde{Q})=x(P-R)\ll x(R),
\end{equation}
where the lower bound follows from applying Lemma~\ref{lemma:dblbd} twice, and the upper bound is taken from~\eqref{eq:sumlbsim}. The upper bound does not depend on $\lambda$ as we have assumed that $\lambda$ is bounded from above.
\subsection{Height estimates}
Since $P=4\tilde{Q}+R$, we obtain from the triangle inequality
\begin{equation}\label{eq:triangle}
4\sqrt{\h(\tilde{Q})}-\sqrt{\h(R)}\leq \sqrt{\h(P)}\leq 4\sqrt{\h(\tilde{Q})}+\sqrt{\h(R)}.
\end{equation}
Apply the assumption  $\log D<\delta h(P)$ to \eqref{eq:height}, we obtain
\begin{equation}\label{eq:triP}
(1-\delta)h(P)-O(1)\leq \h(P)\leq h(P) +O(1).
\end{equation} 
Similarly using \eqref{eq:height} and $h(R)<\lambda h(P)$, we have
\begin{equation}\label{eq:triR}
\h(R)\leq h(R)+O(1)\leq \lambda h(P)+O(1).
\end{equation} 
On the other hand, by \eqref{eq:heightnf} and the assumption $\log D<\delta h(P)$,
\begin{equation}\label{eq:triQ}
h(\tilde{Q})-\delta h(P)-O(1)\leq \h(\tilde{Q})\leq h(\tilde{Q}) +\delta h(P)+O(1).
\end{equation} 
Applying \eqref{eq:triP} and \eqref{eq:triR} to \eqref{eq:triangle}, then squaring, we have
\begin{equation}\label{eq:PQest}
(\sqrt{1-\delta}-\sqrt{\lambda})^2-16\delta
-o(1)
\leq \frac{16h(\tilde{Q})}{h(P)}\leq
(1+\sqrt{\lambda})^2+16\delta
+o(1).
\end{equation}

\subsection{Approximation of algebraic numbers}
If $S\in \E_D(\Qbar)$ such that $4S=R$, write
\[x(R)=x(4S)=\frac{\phi_4(S)}{\psi_4(S)^2},\]
where $\psi_n$ is the $n$th-division polynomial of $\E_D$ and $\phi_n=x\psi_n^2-\psi_{n+1}\psi_{n-1}$.
Define
\begin{equation}\label{eq:f_R}
f_R(T)\coloneqq \prod_{S:4S=R}(T-x(S))=\phi_4(T)-x(R)\psi_4(T)^2.
\end{equation}
Note that every elliptic curve has $16$ $4$-torsion points over $\overline{\Q}$. The two expressions in \eqref{eq:f_R} are equivalent as they are both monic polynomials of degree $16$ with roots $\{x(S)\in\Qbar:4S=R\}$. 
Put $T=x(\tilde{Q})$, then
\[\prod_{4S=R}(x(\tilde{Q})-x(S))=\phi_4(\tilde{Q})-x(R)\psi_4(\tilde{Q})^2.\]
Substitute 
$\phi_4(\tilde{Q})=\psi_4(\tilde{Q})^2 x(4\tilde{Q})$, we get
\begin{equation}\label{eq:multiple}
\prod_{4S=R}(x(\tilde{Q})-x(S))=\psi_4(\tilde{Q})^2(x(4\tilde{Q})-x(R)).
\end{equation}

Now 
\begin{align*}
x(P)&=x(4\tilde{Q}+R)=\left(\frac{y(4\tilde{Q})-y(R)}{x(4\tilde{Q})-x(R)}\right)^2-x(4\tilde{Q})-x(R)\\
&=\frac{-y(4\tilde{Q})y(R)
+x(4\tilde{Q})^2x(R)+x(4\tilde{Q})x(R)^2
-D^2(x(4\tilde{Q})+x(R))
}
{(x(4\tilde{Q})-x(R))^2}
.
\end{align*}
Using~\eqref{eq:multiple}, we have
\begin{multline}\label{eq:xPpoly}
x(P)\left(\prod_{4S=R}(x(\tilde{Q})-x(S))\right)^2\\
=\psi_4(\tilde{Q})^4\left(-y(4\tilde{Q})y(R)
+x(4\tilde{Q})^2x(R)+x(4\tilde{Q})x(R)^2
-D^2(x(4\tilde{Q})+x(R))
\right).
\end{multline}
If $(x,y)\in \E_D(\overline{\Q})$ and $x\in\R$, then we must have $y^2\ll\max\{x,D\}^{3}$, so
\[\psi_4((x,y))^4=(4y(x^6-5D^2x^4-5D^4x^2+D^6))^4\ll\max\{x,D\}^{4\cdot \frac{15}{2}},\]
and we can bound \eqref{eq:xPpoly} by
\[\ll
x(4\tilde{Q})x(R)^2\max\{x(\tilde{Q}),D\}^{30}
\ll x(R)^{33}
\ll x(P)^{33\lambda}
,\]
where we have applied \eqref{eq:tQsize}.
Taking logs,
\begin{equation}\label{eq:prod}
\frac{\log\prod_{4S=R}|x(\tilde{Q})-x(S)|}{h(P)}\leq -\frac{1}{2}
+\frac{33}{2}\lambda+O\left(\frac{\delta}{\log D}\right).
\end{equation}

Let $\alpha=x(S)$ be a root of $f_R$.
Apply~\cite[p.262 last line]{Mahler}, 
\[|f_R'(\alpha)|\gg |\Delta(f_R)|^{1/2}\|f_R\|_{1}^{-14},\]
where $\Delta(\cdot)$ denotes the discriminant and $\|\cdot\|_{1}$ denotes the $\ell_1$-norm.
Write $x(R)=\frac{r}{s}$, where $\gcd(r,s)=1$. 
Since $sf_R(T)\in\Z[T]$, so
$|\Delta(f_R)|\geq s^{-30}$.
Also we can check that
$\|f_R\|_{1}\ll D^{14}\max\{x(R),D^2\}$.
Therefore noting that $H(R)=r\geq Ds$,
\[
\prod_{\tilde{S}\neq S:4\tilde{S}=R}|x(\tilde{S})-x(S)|
=
|f_R'(\alpha)|
\gg
|s|^{-15}(D^{14}\max\{x(R),D^2\})^{-14}
\geq
H(R)^{-15}D^{-209}
.
\]
Take $S\in \{\tilde{S}\in \E_D(\Qbar):4\tilde{S}=R\}$ such that $|x(\tilde{Q})-x(S)|$ is minimum.
By the triangle inequality
\[|x(S)-x(\tilde{S})|\leq|x(\tilde{Q})-x(S)|+|x(\tilde{Q})-x(\tilde{S})|
\leq 2|x(\tilde{Q})-x(\tilde{S})|.
\]
Taking products
\[
\prod_{\tilde{S}\neq S:4\tilde{S}=R}
|x(\tilde{Q})-x(\tilde{S})|
\gg
\prod_{\tilde{S}\neq S:4\tilde{S}=R}|x(S)-x(\tilde{S})|
\gg
H(R)^{-15}D^{-209}
.\]
Take logs
\[\log \prod_{\tilde{S}\neq S:4\tilde{S}=R}
|x(\tilde{Q})-x(\tilde{S})|\geq
-15h(R)-209\log D+O(1).\]
Put this back to~\eqref{eq:prod},
\[
\frac{\log|x(\tilde{Q})-x(S)|}{h(P)}\leq -\frac{1}{2}
+\frac{63}{2}\lambda+209\delta+o(1)
.
\]
Applying the upper bound in~\eqref{eq:PQest} proves Lemma~\ref{lemma:Roth}.

\section{Roth's Theorem}
In this section we follow the proof of Roth's Theorem in Chapter~6 of~\cite{BG}, specialising in the bivariate case.

Let $K\subseteq E$ be number fields such that $m\coloneqq [E:K]$. Suppose $\alpha\in E$. Let $v_1$ be the infinite place of $E$ given by the embedding $\iota|_E:E\hookrightarrow \C$, where $\iota:\Qbar\hookrightarrow \C$ is as fixed in Section~\ref{sec:largept}, so we can write $|\cdot|=|\cdot|_{v_1}$.
We call $\beta\in K$ a $K$-\emph{approximation to $\alpha$ with exponent $\kappa$}, if 
\[|\beta-\alpha|<H(\beta)^{-\kappa}.\]

Approximations obey the following strong gap principle.
\begin{theorem}[strong gap principle {\cite[Theorem~6.5.4]{BG}}]\label{th:gap}
Let $\beta, \beta'\in K$ be distinct elements such that 
$|\alpha-\beta|<H(\beta)^{-\kappa}$,
$|\alpha-\beta'|<H(\beta')^{-\kappa}$ and
$h(\beta')\geq h(\beta)$. Then
\[h(\beta')\geq -2\log 2+(\kappa-1)h(\beta).\]
\end{theorem}

\begin{proof}
We have
\begin{multline*}
\log |\beta-\beta'|=\log|(\alpha-\beta')-(\alpha-\beta)|
\leq \max\left(\log|\alpha-\beta'|,\log |\alpha-\beta|\right)+\log 2\\
\leq -\kappa\min\left(h(\beta'),h(\beta)\right)+\log 2
= -\kappa h(\beta)+\log 2.
\end{multline*}
Also
\[
\log |\beta-\beta'|\geq -h(\beta-\beta')
\geq -h(\beta)-h(\beta')-\log 2.
\]
Combining the upper and lower bounds of $\log |\beta-\beta'|$ gives the required inequality.
\end{proof}

\begin{theorem}\label{theorem:indp}
Let $c<1$, $M\geq 72$ and $L= (\frac{h(\alpha)+\log 2}{c^{-2}-1}+4)M$.
Assume \begin{equation}\label{eq:kappaMc}
\kappa> \left(c-4\sqrt{\frac{m}{M}}\right)^{-1}\left(1+\frac{c^{-2}+1}{M}\right)\sqrt{2m}.\end{equation}
Suppose $\beta_1,\beta_2\in K$ are both approximations to $\alpha\in E$ with exponent $\kappa$.
If $h(\beta_1)\geq L$, then $h(\beta_{2})< Mh(\beta_1)$.
\end{theorem}
We prove Theorem~\ref{theorem:indp} by contradiction. Suppose we can find $\beta_1,\beta_2$ under the assumptions in Theorem~\ref{theorem:indp}, and such that $h(\beta_1)\geq L$ and $h(\beta_{2})\geq Mh(\beta_1)$.
Let $\sigma\coloneqq \sqrt{\frac{2}{M}}$ and $t\coloneqq c\sqrt{\frac{2}{m}}$.

\subsection{The auxiliary polynomial}
Take $N$ large.
Choose
\[d_j=\left\lfloor \frac{N}{h(\beta_j)}\right\rfloor
\text{ for } j=1,2.\]
Let $t<1$, $\balpha=(\alpha,\alpha)\in E^2$ and $\bbeta=(\beta_1,\beta_2)\in K^2$.
Let
\[V_2(t)\coloneqq \vol\left(\left\{(x_1,x_2):x_1+x_2\leq t,\ 0\leq x_j\leq 1\right\}\right)=\frac{1}{2}t^2.\]
For a polynomial $F(x_1,x_2)=\sum_{\mathbf{j}} a_{\mathbf{j}}\mathbf{x}^{\mathbf{j}}\in\Qbar[x_1,x_2]$, define $|F|_v=\max_{\mathbf{j}}|a_{\mathbf{j}}|_v$, $H(F)\coloneqq \prod_{v}|F|_v$ and $h(F)=\log H(F)$.

We apply the following lemma to construct an auxiliary polynomial.
\begin{lemma}[{\cite[Lemma~6.3.4]{BG}}]
Suppose $mV_2(t)<1$. Then for all sufficiently large $d_1, d_2\in\Z$, there exist $F\in K[x_1,x_2]$, $F\not\equiv 0$, with partial degrees at most $d_1,d_2$ such that
\[\ind(F;\mathbf{d},\balpha)\coloneqq \min_{\boldsymbol \mu}\left\{\frac{\mu_1}{d_1}+\frac{\mu_2}{d_2}:\partial_{\mu}F(\balpha)\neq 0\right\}\geq t;\]
and
\[h(F)\leq\frac{mV_2(t)}{1-mV_2(t)}\sum_{j=1}^2 (h(\alpha_j)+\log 2+o(1))d_j,\]
as $d_j\rightarrow \infty$.
\end{lemma}
Since $\frac{1}{2}mt^2<1$, we have
\[\frac{mV_2(t)}{1-mV_2(t)}= \frac{mt^2}{2-mt^2}
=\frac{1}{2m^{-1}t^{-2}-1}.\]
Take $C_1\coloneqq \frac{h(\alpha)+\log 2}{c^{-2}-1}$, so $L=(C_1+4)M$.
Then we can obtain a non-trivial polynomial $F\in K[x_1,x_2]$ with partial degrees at most $d_1,d_2$ such that 
\begin{equation}\label{eq:R1}
\ind(F;\mathbf{d},\balpha)\geq t\quad\text{ and }\quad
h(F)<\frac{2C_1N}{L}.
\end{equation}

\subsection{Non-vanishing at the rational point}
Next we apply Roth's lemma to construct a suitable derivative of $F$ that does not vanish at $\beta$.
\begin{lemma}[Roth's lemma {\cite[Lemma~6.3.7]{BG}}]
Let $F\in\Qbar[x_1,x_2]$ with partial degrees at most $ d_1,d_2$ and $F\not\equiv 0$.
Let $(\xi_1,\xi_2)\in\Qbar^2$ and $0<\sigma^2\leq \frac{1}{2}$.
Suppose that 
$d_{2}\leq \sigma^2 d_1$ and $\min_j d_jh(\xi_j)\geq \sigma^{-2}(h(F)+8d_1)$.
Then 
$\ind(F;\mathbf{d},\boldsymbol\xi)\leq 4\sigma$.
\end{lemma}

Since $L\geq 2\sigma^{-2}(C_1+4)$ and $M\geq 2\sigma^{-2}$, we can apply the lemma to get 
$\ind(F;\mathbf{d},\bbeta)\leq 4\sigma$.
Now we can take $\boldsymbol \mu$ such that $\partial_{\boldsymbol \mu}F(\bbeta)\neq 0$ and $\frac{\mu_1}{d_1}+\frac{\mu_2}{d_2}=\ind (F;\mathbf{d},\bbeta)$.
Let $G=\partial_{\boldsymbol \mu}F$.
Since $\ind(G;\mathbf{d},\balpha)\geq \ind(F;\mathbf{d},\balpha)-\frac{\mu_1}{d_1}-\frac{\mu_2}{d_2}$ by~\cite[6.3.2(c)]{BG}, we deduce from~\eqref{eq:R1} that
\begin{equation}\label{eq:indG}
\ind(G;\mathbf{d},\balpha)\geq t-4\sigma,\qquad
G(\beta)\neq 0,\quad\text{and }\quad
h(G)\leq \frac{4C_1N}{L}.\end{equation}

\subsection{The upper bound}
For places $v\neq v_1$, we have
\begin{equation}\label{eq:prime}
\log |G(\bbeta)|_v\leq \log |G|_v+\sum_{j=1}^{2}d_j(\log^+|\beta_j|_v+\varepsilon_v o(1)),
\end{equation}
where $o(1)\rightarrow 0$ as $d_j\rightarrow\infty$, and
\[\varepsilon_v=\begin{cases}
\frac{[K_v:\Q_v]}{[K:\Q]}&\text{ if $v$ is archimedean}\\
0&\text{ if $v$ is non-archimedean}.\end{cases}
\]
For the place $v_1$, expand $G$ in Taylor series with center $\balpha$
\begin{equation}\label{eq:taylor}
G(\bbeta)=\sum_{\mathbf{k}} \partial_{\mathbf{k}}G(\balpha)(\beta_1-\alpha)^{k_1} (\beta_2-\alpha)^{k_2}.\end{equation}
We have from~\eqref{eq:indG}, that
\[\partial_{\mathbf{k}}G(\balpha)=0\quad\text{ if }\quad\frac{k_1}{d_1}+\frac{k_2}{d_2}<t-4\sigma,\]
and
\[\log|\partial_{\mathbf{k}}G(\balpha)|_{v_1}\leq \log |G|_{v_1}+\sum_{j=1}^2(d_j-k_j)\log^+|\alpha|_{v_1}+\varepsilon_{v_1}(\log 2+o(1))d_j,\]
so putting back to~\eqref{eq:taylor} and taking absolute values and logs,
\begin{equation}\label{eq:inftyel}\begin{split}
\log|G(\bbeta)|_{v_1}
&\leq \max_{\mathbf{k}}\log\left| \partial_{\mathbf{k}}G(\balpha) \prod_{j=1}^2(\beta_j-\alpha)^{k_j}
\right|_{v_1}
+\varepsilon_{v_1}\sum_{j=1}^2\log(d_j+1)\\
&\leq 
-\min_{\frac{k_1}{d_1}+\frac{k_2}{d_2}\geq t-4\sigma}
\left(\sum_{j=1}^2 k_j \log^+\frac{1}{|\beta_j-\alpha|_{v_1}}\right)
 +\log |G|_{v_1}\\
&\qquad \qquad\qquad\qquad
+\sum_{j=1}^2\left(\log^+|\beta_j|_{v_1}+\log^+|\alpha|_{v_1}+\varepsilon_{v_1}(\log 2 +o(1))\right)d_j.
\end{split}
\end{equation}

We now suppress the subscript $v_1$, as $|\cdot|$ is defined to be $|\cdot|_{v_1}$.
Adding up the bounds~\eqref{eq:prime} and~\eqref{eq:inftyel} for all places $v$, and noting that $\sum_v\varepsilon_v=1$, we have
\[\begin{split}
\sum_{v} \log |G(\bbeta)|_v
& \leq 
-\min_{\frac{k_1}{d_1}+\frac{k_2}{d_2}\geq t-4\sigma}
\left(\sum_{j=1}^2 k_j \log^+\frac{1}{|\beta_j-\alpha|}\right)
 +h(G)\\
&\qquad \qquad\qquad\qquad
+\sum_{j=1}^2
\left(h(\beta_j)+\log^+|\alpha|+2\log 2 +o(1)\right)d_j
\\ 
& \leq 
-\min_{\frac{k_1}{d_1}+\frac{k_2}{d_2}\geq t-4\sigma}
\left(\sum_{j=1}^2 k_j \log^+\frac{1}{|\beta_j-\alpha|}\right)
+\left(2+\frac{C_2}{L}\right)N+o(N)
,
\end{split}\]
where $C_2=4C_1+4\log 2+2\log^+|\alpha|$.

From the assumption that $\beta_1$ and $\beta_2$ are both $K$-approximation to $\alpha$ with exponent $\kappa$, we see that
\[\kappa h(\beta_j)\leq\log^+\frac{1}{|\beta_j-\alpha|},\]
we have
\[\sum_{j=1}^2 k_j \log^+\frac{1}{|\beta_j-\alpha|}
\geq
\kappa
\sum_{j=1}^2 \left(h(\beta_j)d_j\right)\frac{k_j}{d_j} 
\sim N\kappa
\left(\frac{k_1}{d_1}+\frac{k_2}{d_2}\right)
.\]
This gives us the upper bound
\begin{equation}\label{eq:ublogG}
\sum_{v} \log |G(\bbeta)|_v\leq -\kappa\left(t-4\sigma\right)N+\left(2+\frac{C_2}{L}\right)N+o(N).
\end{equation}

\subsection{Obtaining the bound}
Since $G(\bbeta)\neq 0$, we have 
$\sum_{v}\log |G(\bbeta)|_v=0$. Put this into~\eqref{eq:ublogG} and let $N\rightarrow\infty$, we get
\[-\kappa\left(\frac{t}{2}-2\sigma\right)+1+\frac{C_2}{2L}\geq 0.\]
Since by assumption $\sigma<\frac{1}{6}$, we have
\[\kappa\leq \left(\frac{t}{2}-2\sigma\right)^{-1}\left(1+\frac{C_2}{2L}\right),\]
which contradicts~\eqref{eq:kappaMc}.
The completes the proof of Theorem~\ref{theorem:indp}.

\section{Bounding the number of points}\label{sec:bound}
In this section we prove the explicit upper bound of $\#\mathcal{L}_D(R)$ given in Theorem~\ref{case:positive2}, when $x(R)>D$ and $R\notin \T_D+2\E_D(\Q)$.
Take $R$ to be the point with minimum canonical height in the coset $R+2\E_D(\Q)$.
Let $\epsilon=0.00153$, which satisfies the assumption in Lemma~\ref{lemma:smallptbd}.

For each $\tilde{Q}\in \frac{1}{2}\E_D(\Q)$, define
$L_{\tilde{Q}}\coloneqq  (\frac{h(S)+\log 2}{c^{-2}-1}+4)M$ as in Theorem~\ref{theorem:indp}, where $S$ is chosen in $\frac{1}{4}R$ such that $|x(S)-x(\tilde{Q})|$ is minimum, with absolute constants $M$ and $c$ to be specified later.
We bound the number of medium points
\[\A_1\coloneqq \Big\{P\in \mathcal{L}_D(R): h(\tilde{Q})<L_{\tilde{Q}}\text{ for some }\tilde{Q}\in \frac{1}{4}(P-R)\Big\}\]
and large points
\[\A_2\coloneqq \Big\{P\in \mathcal{L}_D(R): h(\tilde{Q})\geq L_{\tilde{Q}}\text{ for all }\tilde{Q}\in \frac{1}{4}(P-R) \Big\}.\]
 For each $S\in \frac{1}{4}R$, define
\[\B_2(S)\coloneqq \Big\{\tilde{Q}\in \frac{1}{2}\E_D(\Q):4\tilde{Q}+R\in \A_2,\ 
|x(S)-x(\tilde{Q})| \text{ minimum over }S\in \frac{1}{4}R\Big\}.\]

\subsection{Medium points}
Let $P_1, P_2,\dots, P_{s}$ be points in $\A_1$ with strictly increasing height.
From \eqref{eq:repulsion}, we deduce that
\[h(P_{i+1})\geq 3h(P_i)-4\log D+O(1)\]
for every $i = 1,\dots, s-1$. 
Iterating this,
\[h(P_{i+1})\geq 3^{i}h(P_1)-(1+3+\dots+3^{i-1})(4\log D+O(1))
=3^{i}h(P_1)-2(3^{i}-1)(\log D+O(1)).\]
Therefore it follows from the assumptions $x(P_1)>D^{2(1+\epsilon)}$ that 
\begin{equation}\label{eq:medgap}
h(P_{s})>\left(1-\frac{2\log D+O(1)}{h(P_1)}\right)
3^{s-1}h(P_{1})
>\left(\frac{\epsilon}{1+\epsilon}-o(1)\right)
3^{s-1}h(P_{1})
.
\end{equation}
Take 
\begin{equation}\label{eq:lamdel}
\lambda>\frac{1+\epsilon}{3^{s-1} \epsilon}\quad \text{ and }\quad \delta>\frac{1}{2\cdot 3^{s-1} \epsilon },
\end{equation}
so that 
$\lambda h(P_s)>h(P_1)>h(R)$
and
$\delta h(P_s)>\frac{1}{2(1+\epsilon)}h(P_1)>\log D$.

For each $S\in \frac{1}{4}R$, since $16\h(S)=\h(R)\leq \h(P_1)$, we have by applying~\eqref{eq:heightnf} to $\h(S)$ and \eqref{eq:height} to $\h(P_1)$, 
\[h(S)<\frac{1}{16}h(P_1)+\log D-o(1).\]
Writing $P_s=4\tilde{Q}_s+R$ and using the lower bound in~\eqref{eq:PQest},
\begin{multline*}
\frac{1}{16}
h(P_s)\left((\sqrt{1-\delta}-\sqrt{\lambda})^2-16\delta
-o(1)\right)
\leq h(\tilde{Q}_s)\\
<L_{\tilde{Q}_s}
< \left(\frac{\frac{1}{16}h(P_1)+\log D+\log 2+o(1)}{c^{-2}-1}+5\right)M
.
\end{multline*}
Now apply~\eqref{eq:medgap} and divide both sides by $\frac{1}{16}h(P_1)$,
\[
\left(\frac{\epsilon}{1+\epsilon}-o(1)\right)
3^{s-1}\left((\sqrt{1-\delta}-\sqrt{\lambda})^2-16\delta
+o(1)\right)
< \left(1+\frac{8}{1+\epsilon}+o(1)\right)\frac{M}{c^{-2}-1}
.
\]
Simplifying we have
\begin{equation}\label{eq:medexp}
3^{s-1}
< \left(1+\frac{9}{\epsilon}\right)\frac{M}{(c^{-2}-1)\left((\sqrt{1-\delta}-\sqrt{\lambda})^2-16\delta
\right)}+o(1)
.
\end{equation}
Let $s_0$ be the maximum integer satisfying~\eqref{eq:medexp}, then taking $s=s_0+1$ would be a contradiction as long as \eqref{eq:lamdel} is satisfied. Therefore $\#\A_1\leq 2s_0$, where the factor of $2$ comes from the possible existence of $-P_1,\dots, -P_{s_0}$ .

\subsection{Large points}
Now fix some $S\in \frac{1}{4}R$ and consider the set $\B_2(S)$.
Let $K$ be the minimal number field containing the $x$-coordinates of all points in $\frac{1}{2}\E_D(\Q)$, and let $E$ be the field $K(x(S))$.

Suppose $\tilde{Q}\in \B_2(S)$.
Fix $\lambda=0.000137$ and $\delta=0.0000684$ so that \eqref{eq:lamdel} holds with $s=15$, and take $\kappa=7.516$ , which satisfies 
\[
\kappa<8\cdot \frac{1
-63\lambda-418\delta}{(1+\sqrt{\lambda})^2+16\delta}+o(1),
\]
then Lemma~\ref{lemma:Roth} implies that $x(\tilde{Q})$ is a $K$-approximation to $x(S)$ with exponent $\kappa$. 
Now we can apply Theorem~\ref{theorem:indp} with $m=[E:K]\leq 4$, $M=276.1$ and $c=0.861$, noting that~\eqref{eq:kappaMc} is satisfied.
Take $\beta_1=x(\tilde{Q})$ such that $h(\tilde{Q})$ is minimum over all $\tilde{Q}\in \B_2(S)$. Then Theorem~\ref{theorem:indp} shows that all points in $\B_2(S)$ must have height in the interval $[h(\beta_1),Mh(\beta_1)]$. By Theorem~\ref{th:gap}, if $t$ is the smallest integer such that
\[
(\kappa-1)^{t}>M,
\]
then $\#\B_2(S)\leq t$. This is achieved by $t=3$. 
There are $16$ choices of $S$, but since $\E_D(\Qbar)[4]\subseteq \frac{1}{2}\E_D(\Q)$, if $x(\Q)$ is a $K$-approximation to $x(S)$, then $x(\Q+T)$ is also a $K$-approximation to $x(S+T)$ for any $T\in \E_D(\Qbar)[4]$.
Therefore $\#\A_2\leq 3$.

Returning to the medium points with our choice of constants, we have $\#\A_1\leq 28$.
If $P\in \E_D(\Z)$ then $-P\in \E_D(\Z)$, so $\#(\A_1\cup\A_2)\leq 30$.

\section{Integral points in other cosets of $2\E_D(\Q)$}
We now prove the upper bounds in Theorem~\ref{thm:elliptic}.

\subsection{Cosets with respect to a non-torsion point}
Here we will treat the case in Theorem~\ref{thm:elliptic}~\eqref{case:negative}, assuming $R\notin \T_D+2\E_D(\Q)$.
Suppose $x(R)<0$.
If $P\in\mathcal{Z}_D(R)$, then $-D<x(P)<0$ and so $\h(P)<\log D+\frac{2}{3}\log 2$ by~\eqref{eq:heightn}. Following the argument in Section~\ref{sec:sphere}, we obtain an upper bound of $A(r+1,\theta)$ where $\sin\frac{\theta}{2}=\frac{1}{2}\sqrt{\frac{\log D-2\log 2}{\log D+\frac{2}{3}\log 2}}=\frac{1}{2}-o(1)$. 
For $D\geq 97353$, applying the following estimate by Rankin gives us an upper bound of $3$. Since non-torsion integral points in comes in pairs of $\pm P$, we can reduce the upper bound to $2$ if $R\notin \T_D+2\E_D(\Q)$.
\begin{theorem}[{\cite[Lemma~2]{Rankincaps}}]
If $\frac{\pi}{4}<\theta<\frac{\pi}{2}$, then
\[
A(r,\theta)
\leq 
\frac{2\sin^2\theta}{2\sin^2\theta-1}
.\]
\end{theorem}

Checking all the integral points in the range $-D<x(P)<0$ on $\E_D$ for each $D< 97353$, we see that the only exceptions are those listed in Theorem~\ref{thm:elliptic}~\eqref{case:negative}.

\subsection{Integral points in $2\E_D(\Q)+\T_D$}
We now prove cases~\eqref{case:infty},~\eqref{case:ntors} and~\eqref{case:ptors} in Theorem~\ref{thm:elliptic}.
We first show that if a rational point has a multiple which is an integral point, then the original point must also be integral.

\begin{lemma}\label{lem:multint}
Suppose $P\in \E_D(\Q)$. 
If $mP\in \E_D(\Z)$ for some integer $m\geq 2$, then $P\in \E_D(\Z)$.
\end{lemma}
\begin{proof}
Suppose $P=mQ$, where $Q\in \E_D(\Q)$. 
We have
\[x(P)=\frac{\phi_m(Q)}{\psi_m(Q)^2},\]
where $\psi_m$ is the $m$th division polynomial, and $\phi_m=x\psi_m^2-\psi_{m+1}\psi_{m-1}$ as usual.
The polynomials $\phi_m(x)$ and $\psi_m(x)^2$ have leading terms $x^{m^2}$ and $m^2x^{m^2-1}$ respectively.
Putting $x(Q)=\frac{u}{v}$ with $\gcd(u,v)=1$, and clearing denominators we have
\[x(Q)=\frac{u^{m^2}+vF(u,v)}{v(m^2u^{m^2-1}+vG(u,v))},\]
for some polynomials $F,G\in\Z[x,y]$.
Therefore $x(P)\in\Z$ implies $v\mid u$, so $v=1$ and $Q$ is also integral.
\end{proof}

We show that $2\E_D(\Q)$ contains no integral points.
\begin{lemma}
Suppose $P\in \E_D(\Q)$ is non-torsion.
Then $2P\not\in \E_D(\Z)$. 
\end{lemma}
\begin{proof}
Suppose $P\in \E_D(\Q)$ and $2P\in \E_D(\Z)$, then $P$ must be an integral point by Lemma~\ref{lem:multint}.
Write $P=(x,y)$, so
\[x(2P)=\left(\frac{x^2+D^2}{2y}\right)^2.\]
Suppose $2P\in \E_D(\Z)$. 
Then 
$4y^2=4x(x+D)(x-D)\mid (x^2+D^2)^2$.
Therefore $x\mid D$ and so $x$ is squarefree. Write $d=-\frac{D}{x}$, and we have $4(d-1)(d+1)\mid x(d^2+1)^2$.
Since we assumed that $P$ is not a torsion point, $x\neq D$ and $-D<x<0$. Suppose $d$ is odd, 
then $(d^2+1)^2\equiv 4\bmod 8$ and $8\mid (d-1)(d+1)$, so $8\mid x$, but this contradicts with $x$ being squarefree.
Now suppose $d$ is even, then $(d^2+1)^2$ is odd, so $4\mid x$, which is also a contradiction.
\end{proof}

We now look at points of the form $2P+(-D,0)$ or $2P+(0,0)$.
\begin{lemma}
Suppose $P\in \E_D(\Q)$.
For each $T\in \{(-D,0),\ (0,0)\}$, we have $2P+T\in \E_D(\Z)$ if and only if $P$ is a torsion point. 
\end{lemma}

\begin{proof}
Notice that $\theta(2P+(0,0))=(-D,-1,D)$
and 
$\theta(2P+(-D,0))=(-2D,-D,2)$.
If $2P+(-D,0)\in \E_D(\Z)$, taking $x(2P+(0,0))=-s^2$, we see that the equation
\[s^2+Dt^2=D\]
is solvable for $s,t\in \Z$.
Similarly if $2P+(0,0)\in \E_D(\Z)$, taking $x(2P+(-D,0))=-Du^2$, then 
\[Du^2+2v^2=D\]
is solvable for $u,v\in \Z$.
The only solutions to each of these equations over the integers are given by $s=0$ and $u^2=1$. This implies that in both cases $P$ is a torsion point.
\end{proof}

The only possible non-torsion integral points in $2\E_D(\Q)+\T_D$ are in $(D,0)+2\E_D(\Q)$ and satisfies the property in the following theorem.

\begin{lemma}\label{lem:doublept}
Then there exists some $P\in \E_D(\Q)$ such that $2P+(D,0)\in \E_D(\Z)$ if and only if
the system
\begin{equation}\label{eq:2Pint}
s^2-1=2Du^2,\qquad s^2+1=2v^2
\end{equation}
is solvable for some $s, u,v\in \Z_{>0}$. Furthermore,~\eqref{eq:2Pint} has at most one solution for each $D$. If a solution $(s,u,v)$ exists, then $x(2P+(D,0))=Ds^2$, 
\begin{equation}\label{eq:sepsilon}
s^2+2uv\sqrt{D}=
(v+u\sqrt{D})^2,
\end{equation}
and $v+u\sqrt{D}$ is the fundamental solution to $v^2-Du^2=1$.
\end{lemma}
\begin{proof}
Note that $\theta(2P+(D,0))=(2,D,2D)$.
If $2P+(D,0)\in \E_D(\Z)$, then writing $x(2P+(D,0))=Ds^2$ finds us a solution $(s,u,v)$ to the system~\eqref{eq:2Pint}. Conversely if~\eqref{eq:2Pint} is solvable, it is easy to check that $Ds^2$ is the $x$-coordinate of an integral point on $\E_D$, and this point must be in the same coset of $2\E_D(\Q)$ as $(D,0)$ since $\theta$ is injective.

If~\eqref{eq:2Pint} is solvable, taking the difference of the two equations in~\eqref{eq:2Pint}, we get
$v^2-Du^2=1$.
From~\eqref{eq:2Pint}, we see that~\eqref{eq:sepsilon} holds, and also $s^4-D(2uv)^2=1$.
Cohn showed that such equation has at most one solution unless $D=1785$. More precisely, the main theorem in~\cite{Cohn} implies that $s^2+2uv\sqrt{D}$ is either $a+b\sqrt{D}$ or $(a+b\sqrt{D})^2$ if $a+b\sqrt{D}$ is the fundamental solution to $v^2-Du^2=1$. This proves the final claim.
\end{proof}

\section{Average number of integral points}
In this section we prove Theorem~\ref{thm:avgint}.
From Theorem~\ref{thm:posint}
\[
\#\E_D(\Z)\ll 4^{\rank \E_D(\Q)}.
\]
Heath-Brown~\cite[Theorem~1]{HBSelmer} proved that\footnote{To be precise, the theorem was only stated for odd $D$, but it is possible to extend the proof to even $D$.}
\[\lim_{N\rightarrow\infty}\frac{1}{\#\D_N}\sum_{D\in\D_N} 2^{k\cdot \rank \E_D(\Q)}\ll_k 1.\]
Therefore
\[
\limsup_{N\rightarrow\infty}\frac{1}{\#\D_N}\sum_{D\in\D_N}(\#\E_D(\Z))^2\ll 1.
\]
Suppose $\mathcal{G}_N\subseteq \D_N$.
By Cauchy–Schwarz inequality,
\[\sum_{D\in \mathcal{G}_N}\#\E_D(\Z)
\leq
\Big(\sum_{D\in \D_N}(\#\E_D(\Z))^2\Big)^{1/2}
(\#\mathcal{G}_N)^{1/2}
\ll (\#\D_N)^{1/2}(\#\mathcal{G}_N)^{1/2}.
\]
Therefore the contribution from any subset $\mathcal{G}_N$ of $\D_N$ of size $o(N)$ to the average of $\#\E_D(\Z)$ over $\D_N$ tends to $0$ as $N\rightarrow\infty$.

Assuming~\eqref{eq:minconj} implies that we only need to consider the contribution from the curves $\E_D$ with rank $0$ or $1$. A theorem by Le Boudec~\cite[Proposition~1]{LeBoudec} shows that
\[
\sum_{D\geq 1}\#\left\{P\in\E_D(\Z): x(P)<\frac{N^2}{(\log N)^{\kappa}}\right\}\ll \frac{N}{(\log N)^{\kappa/2-6}},
\]
where we take $\kappa>12$.
Therefore we can also exclude all $\E_D$ with any integral point $H(P)<\frac{N^2}{(\log N)^{\kappa}}$ since there are $o(N)$ of them.

If $\rank \E_D(\Q)=0$, then there are automatically no non-torsion integral points.
In the following we consider $\E_D$ such that $\rank \E_D(\Q)=1$ and any $P\in\E_D(\Z)\setminus \T_D$ satisfy $H(P)>\frac{D^2}{(\log D)^{\kappa}}$.
This removes the need to consider the points arising from cases~\eqref{case:infty},~\eqref{case:ntors},~\eqref{case:negative} in Theorem~\ref{thm:elliptic}.

Our aim is to prove the following.
\begin{theorem}\label{thm:rank1int}
Assume that $\rank \E_D(\Q)=1$ and that any $P\in\E_D(\Z)\setminus \T_D$ satisfy $H(P)>\frac{D^2}{(\log D)^{\kappa}}$.
Then \[\#(\E_D(\Z)\setminus\T_D)\leq 4.\]
\end{theorem}

We now demonstrate that the integral points that appear in Theorem~\ref{thm:elliptic}~\eqref{case:ptors} are rare for $D\in \D_N$ and do not contribute to the average average in Theorem~\ref{thm:avgint}. Recall that these points are classified in Lemma~\ref{lem:doublept}. 
Dirichlet class number formula for real quadratic number fields states that the class number of $\Q(\sqrt{D})$ equals
\[\frac{\sqrt{D}L(1,\chi_D)}{\log \epsilon_D},\]
where $\chi_D$ is the Kronecker symbol $\leg{D}{\cdot}$ and $\epsilon_D$ is the fundamental unit of $\Q(\sqrt{D})$. Since the class number is at least $1$, 
this gives an inequality
\[\log \epsilon_D\leq \sqrt{D}L(1,\chi_D).\]
It is well-known that $L(1,\chi_D)\ll \log D$.
Therefore together with~\eqref{eq:sepsilon}, we have
\begin{equation}\label{eq:sub}
\log s <2\log \epsilon_D\ll \sqrt{D}\log D.
\end{equation}
On the other hand, since $s^2-2v^2=-1$ and $1+\sqrt{2}$ is the fundamental unit of $\Q(\sqrt{2})$, we the possible values of $s$ is given by
\[s=\frac{1}{2}\left((1+\sqrt{2})^k+(1-\sqrt{2})^{k})\right),\]
where $k$ is any positive odd integer.
For large values of $k$, $|(1-\sqrt{2})^{k}|$ is bounded, so
\begin{equation}\label{eq:slb}
s\gg (1+\sqrt{2})^k.
\end{equation}
Putting together the inequalities~\eqref{eq:sub} and~\eqref{eq:slb}, we get 
\[
k\ll \sqrt{D}\log D.\]
Therefore for $D\in \D_N$, there are $\ll\sqrt{N}\log N$ integral points of the form $2P+(D,0)$, which does not contribute to the average in Theorem~\ref{thm:avgint}.

\subsection{Odd multiples of a generator}
It now remains to treat the points not covered by Theorem~\ref{thm:elliptic}.
Notice that if $m$ is odd, $mP+T=m(P+T)$ for any $P\in\E_D(\Q)$ and $T\in \T_D$, so any integral points not in $2\E_D(\Q)+\T_D$ are odd multiples of a generator of the free part of $\E_D(\Q)$.
By Lemma~\ref{lem:multint}, if $mP\in \E_D(\Z)$ then $P\in \E_D(\Z)$.

If $P\in \E_D(\Z)$ and $H(P)>\frac{N^2}{(\log N)^{\kappa}}$, then $x(P+(0,0))$, $x(P+(-D,0))$, $x(P+(D,0))\ll D$, therefore by assumption $P+(0,0)$, $P+(-D,0)$, $P+(D,0)\notin \E_D(\Z)$. 
Therefore it is enough to consider odd multiples of one integral point that is also generator.

We show that small multiples of a reasonably sized rational point, as assumed in Theorem~\ref{thm:rank1int} which we wish to prove, cannot be integral.
\begin{theorem}\label{thm:smallmult}
Let $\kappa>0$ and $C_1<\sqrt{\frac{4}{3}\log 2}$. Suppose $D$ is some sufficiently large squarefree integer, $P\in \E_D(\Q)$ and assume $x(P)>\frac{D^2}{(\log D)^{\kappa}}$, then $mP\notin \E_D(\Z)$ for all $1<m\leq \exp(C_1\sqrt{\log D})$.
\end{theorem}
We have shown that $2P$ cannot be integral, so assume $m\geq 3$. 
With the formulae
\begin{equation}\label{eq:psiodd}
\psi_{2m+1}=\psi_{m+2}\psi_m^3-\psi_{m-1}\psi_{m+1}^3,
\end{equation}
\begin{equation}\label{eq:psieven}
\psi_{2m}=\frac{\psi_m}{2y}\left(\psi_{m+2}\psi_{m-1}^2-\psi_{m-2}\psi_{m+1}^2\right),
\end{equation}
we prove the following by induction.
\begin{lemma}\label{lemma:divp}
Fix some $C_2>\frac{3}{2\log 2}$. 
Let $x>D$ such that
$(x,y)\in \E_D(\Q)$.
Then for any positive integer $m$ satisfying $C_2(\log m)^2<2(\log x-\log D)$, we have
\[\psi_m(x)>\left(1-\exp\left(C_2(\log m)^2\right)\left(\frac{D}{x}\right)^2\right)mx^{\frac{m^2-1}{2}}.\]
\end{lemma}

\begin{proof}
Write $\psi_m(x)=(1-E_m(\frac{D}{x})^2)mx^{\frac{m^2-1}{2}}$. Assuming $E_m\left(\frac{D}{x}\right)^2<1$, we obtain from~\eqref{eq:psiodd}
\begin{equation}\label{eq:Eodd}
E_{2m+1}<E_{m-1}+3E_{m+1}+\frac{m^3(m+2)}{2m+1}(E_{m+1}+3E_m),
\end{equation}
from~\eqref{eq:psieven}
\begin{equation}\label{eq:Eeven}
E_{m+2}<\frac{1}{2}+E_m+\frac{(m - 1)^2 (m + 2)}{4}(E_{m-1}+2E_{m+2})+\frac{1}{2}(E_{m-2}+2E_{m+1}).
\end{equation}
Assuming $E_{m}<\exp\left(C_2(\log m)^2\right)$ for all $m<N$, we obtain an upper bound for $E_N<\exp\left(C_2(\log N)^2\right)$ from~\eqref{eq:Eodd} and~\eqref{eq:Eeven}.
Checking the base cases $\psi_2=2x^{3/2}(1-(\frac{D}{x})^2)^{1/2}$ and $\psi_3=3x^4(1-2(\frac{D}{x})^2-\frac{1}{3}(\frac{D}{x})^4)$ completes the induction.
\end{proof}

Write uniquely $x(mP)=\frac{u}{v_m^2}$, where $\gcd(u,v_m)=1$ and $v_m>0$. 
By~\cite[Lemma~11.4]{Stange} 
\begin{equation}\label{eq:Stange}
\log v_m\leq \log|\psi_m(x)|\leq \log v_m+\frac{1}{8}m^2\log |\Delta_D|,
\end{equation}
where $\Delta_D=(2D)^6$ is the discriminant of $\E_D$.

\begin{proof}[Proof of Theorem~\ref{thm:smallmult}]
Let $x\coloneqq x(P)>\frac{D^2}{(\log D)^{\kappa}}$.
Suppose $mP\in\E_D(\Z)$, then~\eqref{eq:Stange} reduces to
\begin{equation}\label{eq:psiub}
\psi_m(x)\leq (2D)^{\frac{3}{4}m^2}.
\end{equation}

Fix $\epsilon>0$ such that
\[\log m<\sqrt{\frac{1}{C_2}\left(\log(1-\epsilon)+2\log D-2\kappa\log \log D\right)},\]
then by Lemma~\ref{lemma:divp},
\[\psi_m(x)>\epsilon mx^{\frac{m^2-1}{2}}>\epsilon m\left(\frac{D}{(\log D)^{\kappa/2}}\right)^{m^2-1},\]
which contradicts~\eqref{eq:psiub} for sufficiently large $D$. 
\end{proof}

Now following Section~\ref{sec:bound}, we have $\#\A_1=0$ for the medium points using Theorem~\ref{thm:smallmult}, and $\#\A_2\leq 3$ for the large points. Since non-torsion integral points come in pairs $\pm P$, $\#(\A_1\cup\A_2)\leq 2$.
Therefore the possible points contributing to the upper bound in Theorem~\ref{thm:rank1int} comes from the generator and its corresponding negative point, together with the pair of large points in $\#\A_2$.

\section*{Acknowledgements}
The author would like to thank Andrew Granville for useful discussions throughout this project.
The author is supported by the European Research Council grant agreement No.\ 670239.

\end{document}